\documentclass[10pt]{amsart}

\usepackage{amssymb}

\newtheorem{theorem}{Theorem}[section]
\newtheorem{corollary}[theorem]{Corollary}
\newtheorem{lemma}[theorem]{Lemma}

\newtheorem{example}[theorem]{Example}

\numberwithin{equation}{section}

\begin{document}

\title{The Katok-Spatzier Conjecture and Generalized Symmetries}


\author{Lennard F. Bakker}

\email{\rm bakker@math.byu.edu}


\begin{abstract} Within the smooth category, an intertwining is exhibited between the global rigidity of irreducible higher-rank ${\mathbb Z}^d$ Anosov actions on $\mathbb T^n$ and the classification of equilibrium-free flows on ${\mathbb T}^n$ that possess nontrivial generalized symmetries.
\end{abstract}

\maketitle

\markboth{\sc Spatzier's Conjecture and Generalized Symmetries}{\sc L.F. Bakker}

\section{Introduction}

In \cite{GO}, Spatzier communicated the conjecture that any irreducible higher-rank ${\mathbb Z}^d$ Anosov action is $C^\infty$-conjugate to an algebraic action. Later, in \cite{KS}, Kalinin and Spatzier stated a refinement of this conjecture that contends that any irreducible higher-rank ${\mathbb Z}^d$ Anosov action on any compact manifold has a finite cover $C^\infty$-conjugate to an algebraic action. The asserted global rigidity was motivated in part by earlier results of Katok and Lewis in \cite{KL1} and \cite{KL2},  and a more recent result by Rodriguez Hertz in \cite{RH}. In the latter, global rigidity has been shown for any higher-rank ${\mathbb Z}^d$ Anosov action on ${\mathbb T}^n$ whose action on homology has simple eigenvalues and whose course Lyapunov spaces are one or two dimensional, plus additional conditions. A partial confirmation of the refined assertion of global rigidity is provided in \cite{KS} for  higher-rank ${\mathbb Z}^d$ Anosov $C^\infty$ actions each of whose course Lyapunov spaces are one-dimensional, plus additional conditions. If a higher-rank ${\mathbb Z}^d$ Anosov action on ${\mathbb T}^n$ is $C^\infty$-conjugate to an algebraic one and that algebraic action has a  common real eigenvector, then that higher-rank ${\mathbb Z}^d$ Anosov action preserves a  one-dimensional $C^\infty$ foliation of ${\mathbb T}^n$ determined by that common real eigenvector, i.e., generated by a equilibrium-free $C^\infty$ flow.

This paper intertwines the global rigidity of  higher-rank ${\mathbb Z}^d$ Anosov $C^\infty$ actions on ${\mathbb T}^n$ with the classification of equilibrium-free $C^\infty$ flows on ${\mathbb T}^n$ that possess nontrivial generalized symmetries. The intertwining centers on presence of a single  one-dimensional $C^\infty$ distribution determined by an equilibrium-free $C^\infty$ flow that is invariant under a ${\mathbb Z}^d$ Anosov $C^\infty$ action, without a priori conditions on all the course Lyapunov spaces. As shown in Section 3, any generalized symmetry of an equilibrium-free flow is nontrivial if it is Anosov (see Corollary \ref{multiplierAnosov}). Furthermore, any equilibrium-free flow that possesses a nontrivial generalized symmetry does not have any uniformly hyperbolic compact invariant sets (see Corollary \ref{nonhyper}). The intertwining juxtaposes an equilibrium-free $C^\infty$ flow that is not Anosov with a smooth ${\mathbb Z}^d$ action that is Anosov. In the $C^\infty$ topology, this is a counterpoint to the result of Palis and Yoccoz in \cite{PY} on the triviality of centralizers for an open and dense subset of Anosov diffeomorphisms on ${\mathbb T}^n$, and also to the result of Sad in \cite{Sa} on the local triviality of centralizers for an open and dense subset of Axiom A vector fields that satisfy the strong transversality condition (as applied to vector fields on ${\mathbb T}^n$). In particular, it is quite rare for an equilibrium-free $C^\infty$ flow on ${\mathbb T}^n$ (or more generally, on a closed Riemannian manifold) to possess a nontrivial generalized symmetry (see Corollary \ref{rare}).

The first aspect of the intertwining on ${\mathbb  T}^n$ relates the global rigidity of a ${\mathbb Z}^d$ Anosov $C^\infty$ action with an equilibrium-free $C^\infty$ flow that is quasiperiodic. As detailed in Section 2, the generalized symmetry group of a $C^\infty$ flow $\Phi$ is the subgroup $S_\Phi$ of ${\rm Diff}^\infty({\mathbb T}^n)$ each of whose elements $R$ sends (via the pushforward) the generating vector field $X_\Phi$ of $\Phi$ to a uniform scalar multiple $\rho_\Phi(R)$ of itself. The multiplier group $M_\Phi$ of $\Phi$ is the abelian group of these scalars. As shown in Section 4, the elements of $M_\Phi\setminus\{1,-1\}$ when $\Phi$ is quasiperiodic (or more generally, minimal) are algebraic integers of degree between $2$ and $n$ inclusively (see Corollary \ref{algebraicnature}).

\begin{theorem}\label{Anosovalgebraic} On ${\mathbb T}^n$, suppose $\alpha$ is a\, ${\mathbb Z}^d$ Anosov $C^\infty$ action, and $\Phi$ is equilibrium-free $C^\infty$ flow. If $\alpha({\mathbb Z}^d)\subset S_\Phi$ and $\Phi$ is quasiperiodic $($i.e., $C^\infty$-conjugate to an irrational flow$)$, then $\alpha({\mathbb Z}^d)$ is $C^\infty$-conjugate to an affine action, a finite index subgroup of $\alpha({\mathbb Z}^d)$ is\, $C^\infty$-conjugate to an algebraic action, and $M_\Phi$ contains a ${\mathbb Z}^d$ subgroup.
\end{theorem}

\noindent Relevant definitions and the proof are given in Section 6. The proof holds not only for $d\geq 2$ but also for $d=1$. It uses a semidirect product characterization of the structure of the generalized symmetry group for an irrational flow (as shown in Section 5). It also uses the existence of a common fixed point for a finite index subgroup of the ${\mathbb Z}^d$ Anosov action, a device used in other global rigidity results (for example, see \cite{KL2}).

The second aspect of the intertwining on ${\mathbb T}^n$ relates the classification of an equilibrium-free $C^\infty$ flow with a ${\mathbb Z}^d$ Anosov $C^\infty$ action that is topologically irreducible. As detailed in Section 7, an irrational flow $\phi$ on ${\mathbb T}^n$ is of Koch type if a uniform scalar multiple of its frequencies form a ${\mathbb Q}$-basis for a real algebraic number field of degree $n$ (also see \cite{KO} and \cite{LD}). For an $R\in S_\Phi$ of an equilibrium-free $C^\infty$ flow $\Phi$, the quantity $\log\vert \rho_\Phi(R)\vert$ is the value of the Lyapunov exponent $\chi_R$ of $R$ in the direction of $X_\Phi$ (see Theorem \ref{Lyapunov}).

\begin{theorem}\label{flowKoch} On ${\mathbb T}^n$, suppose $\alpha$ is a higher-rank\, ${\mathbb Z}^d$ Anosov $C^\infty$ action, and $\Phi$ is equilibrium-free $C^\infty$ flow. If $\alpha({\mathbb Z}^d)\subset S_\Phi$, and $\alpha$ is topologically irreducible and $C^\infty$-conjugate to an algebraic ${\mathbb Z}^d$ action, and for an Anosov element $R\in{\alpha({\mathbb Z}^d})$, the multiplicity of the value $\log\vert \rho_\Phi(R)\vert$ of $\chi_R$ is one at some point of\, ${\mathbb T}^n$, then $\Phi$ is  projectively $C^\infty$-conjugate to an irrational flow of Koch type.
\end{theorem}

\noindent Relevant definitions and the proof are given in Section 7. The proof uses the Oseledets decomposition for an Anosov diffeomorphism (see \cite{BP} and \cite{KH}) to show that the flow is $C^\infty$-conjugate to one generated by a constant vector field. Then by the topological irreducibility and results of Wallace in \cite{WA}, the components of a scalar multiple of the constant vector field are shown to form a ${\mathbb Q}$-basis for a real algebraic number field.

\section{Flows with Nontrivial Generalized Symmetries}

Generalized symmetries extend the classical notions of time-preserving and time-reversing symmetries  of flows. To simplify notations for these and for proofs of results, it is assumed throughout the remainder of the paper that all manifolds, flows, vector fields, diffeomorphisms, distributions, etc., are smooth, i.e., of class $C^\infty$. Let $P$ be a closed (i.e., compact without bounday) manifold. Let ${\rm Flow}(P)$ denote the set of flows on $P$. Following \cite{BC}, a {\it generalized symmetry}\, of $\psi\in{\rm Flow}(P)$ is an $R\in{\rm Diff}(P)$ such that there is $\mu\in{\mathbb R}^\times = {\mathbb R}\setminus\{0\}$ (the multiplicative real group) for which
\[ R\psi(t,p) = \psi(\mu t, R(p)){\rm\  for\ all\ } t\in{\mathbb R}{\rm\ and\ all\ } p\in P.\]
It is easy to show that $R$ being a generalized symmetry of $\psi$ is equivalent to $R$ satisfying 
\[ R_*X_\psi = \mu X_\psi{\rm\ for\ some\ }\mu\in{\mathbb R}^\times.\]
Here $X_\psi(p)=(d/dt)\psi(t,p)\vert_{t=0}$ is the vector field that generates $\psi$, and $R_*X_\psi = {\bf T}R X_\psi R^{-1}$ is the push-forward of $X_\psi$ by $R$ where ${\bf T}R$ is the derivative map. The {\it generalized symmetry group}\, of $\psi$ is the set $S_\psi$ of all the generalized symmetries that $\psi$ possesses. There is a homomorphism $\rho_\psi:S_\psi \to {\mathbb R}^\times$ taking $R\in S_\psi$ to its unique multiplier $\rho_\psi(R) = \mu$. The multiplier group of $\psi$ is $M_\psi = \rho_\psi(S_\psi)$.

The generalized symmetry group and the multiplier group of a flow are invariants for the equivalence relation of projective conjugacy. Two $\psi,\phi\in{\rm Flow}(P)$ are {\it projectively conjugate}\, if there are $h\in{\rm Diff}(P)$ and $\vartheta\in{\mathbb R}^\times$ such that $h_*X_\psi = \vartheta X_\phi$. Projective conjugacy is an equivalence relation on ${\rm Flow}(P)$. Projective conjugacy reduces to smooth conjugacy when $\vartheta = 1$. For $h\in{\rm Diff}(P)$, let $\Delta_h$ be the inner automorphism of ${\rm Diff}(P)$ given by $\Delta_h(R) = h^{-1}Rh$ for $R\in{\rm Diff}(P)$. If $h_*X_\psi = \vartheta X_\phi$, then $\Delta_h(S_\phi) = S_\psi$ (see Theorem 4.1 in \cite{BC} which states that $S_\psi$ is conjugate to the generalized symmetry group for the flow determined by $\vartheta X_\phi$, which is exactly the same as $S_\phi$.) Furthermore, if $\psi$ and $\phi$ are projectively conjugate, then $M_\psi=M_\phi$ (see Theorem 2.2 in \cite{BA5}), i.e., the multiplier group is an absolute invariant of projective conjugacy.

Any $R\in S_\psi$ is a trivial generalized symmetry of $\psi$ if  $\rho_\psi(R)=1$ (i.e., $R$ is time-preserving), or if $\rho_\psi(R)=-1$ (i.e., $R$ is time-reversing). Any $R\in S_\psi$ with $\vert\rho_\psi(R)\vert\ne 1$ is a {\it nontrivial generalized symmetry}\, of $\psi$. A flow $\psi$ (or equivalently its generating vector field $X_\psi$) is said to possess a nontrivial generalized symmetry when $M_\psi\setminus\{1,-1\}\ne\emptyset$. 

\begin{theorem}\label{noperiodic} Let $\psi$ be a flow on a closed Riemannian manifold $P$. If $\psi$ has a periodic orbit and $M_\psi\setminus\{1,-1\}\ne\emptyset$, then $\psi$ has a nonhyperbolic equilibrium.
\end{theorem}

\begin{proof} Suppose for $p_0\in P$ that ${\mathcal O}_\psi(p_0)=\{\psi_t(p_0):t\in{\mathbb R}\}$ is a periodic orbit whose fundamental period is $T_0>0$. (Here $\psi_t(p) = \psi(t,p)$.) Assuming that $M_\phi\setminus\{1,-1\}\ne\emptyset$ implies there is $Q\in S_\psi$ such that $\rho_\psi(Q)\ne \pm 1$. If $\vert \rho_\psi(Q)\vert >1$, then $\vert\rho_\psi(Q^{-1})\vert<1$, since $\rho_\psi$ is a homomorphism. Hence there is $R\in S_\psi$ such that $\vert\rho_\psi(R)\vert<1$. Then 
\[ R(p_0) = R\psi(T,p_0) = \psi\big(\rho_\psi(R)T_0,R(p_0)\big).\]
For $p_1= R(p_0)$, this implies that ${\mathcal O}_\psi(p_1)$ is a periodic orbit with period $T_1 = \vert \rho_\psi(R)\vert T_0$. Suppose that $T_1$ is not a fundamental period for ${\mathcal O}_\psi(p_1)$, i.e., there is $0<T_1^\prime<T_1$ such that $p_1 = \psi(T_1^\prime,p_1)$. This implies that
\[ R(p_0) = p_1 = \psi(T_1^\prime,p_1) = \psi\big(T_1^\prime,R(p_0)\big) = R\psi(T_1^\prime/\rho_\psi(R),p_0).\]
Since $R$ is invertible, this gives $p_0 = \psi(T_1^\prime/\rho_\psi(R), p_0)$. With $T_0$ as the fundamental period for ${\mathcal O}_\psi(p_0)$, there is $m\in{\mathbb Z}^+$ such that $\vert T_1^\prime/\rho_\psi(R)\vert = mT_0$. Hence
\[ T_1^\prime = m T_0 \vert \rho_\psi(R)\vert \geq T_0\vert \rho_\psi(R)\vert = T_1,\]
a contradiction to $T_1^\prime < T_1$. Thus, $T_1$ is the fundamental period for ${\mathcal O}_\psi(p_1)$. Since $\vert \rho_\psi(R)\vert<1$, it follows that ${\mathcal O}_\psi(p_0)$ and ${\mathcal O}_\psi(p_1)$ have different fundamental periods, and hence are distinct periodic orbits. Iteration gives a sequence of distinct periodic orbits ${\mathcal O}_\psi(p_i)$, where $p_i = R^i(p_0)$, whose fundamental periods $T_i = \vert \rho_\psi(R)\vert^i T_0$ decrease to $0$ since $\vert \rho_\psi(R)\vert <1$.

Let $l(p_i)$ be the arc length of ${\mathcal O}_\psi(p_i)$. In terms of the Riemannian norm $\Vert\cdot\Vert$ on ${\bf T}P$, the tangent bundle of $P$,
\[ l(p_i) = \int_0^{T_i} \Vert X_\psi\big(\psi(t,p_i)\big)\Vert \ dt.\]
Continuity of $X_\psi$ on the compact manifold $P$ implies that $M=\sup_{p\in P} \Vert X_\psi(p)\Vert$ is finite. Thus 
\[ l(p_i)\leq MT_i \to 0.\]
By the compactness of $P$, there is a convergence subsequence $p_{i_m}$ with limit $p_\infty$. If $\Vert X_\psi(p_\infty)\Vert\ne 0$, then the smoothness of $X_\psi$ implies by the Flow Box Theorem that there are no periodic orbits of $\psi$ completely contained in a sufficiently small neighborhood $U$ of $p_\infty$. But since $l(p_{i_m})\to0$ and $p_{i_m}\to p_\infty$, there are periodic orbits of $\psi$ completely contained in $U$, a contradiction. Thus $p_\infty$ is an equilibrium for $\psi$. If $p_\infty$ were hyperbolic, then the Hartman-Grobman Theorem would imply that there are no periodic orbits completely contained in a sufficiently small neighborhood of $p_\infty$, again a contradiction. Therefore, $\phi$ has a nonhyperbolic equilibrium.
\end{proof}

The flows on a closed Riemannian manifold which possess nontrivial generalized symmetries are rare or non-generic in the sense of Baire category theory. Let ${\mathcal X}(P)$ denote the set of vector fields on $P$. With $P$ compact, there is a one-to-one correspondence between ${\mathcal X}(P)$ and ${\rm Flow}(P)$. Equipped with the usual $C^\infty$ topology, ${\mathcal X}(P)$ is a Baire space. A residual subset of ${\mathcal X}(P)$ is one that contains a countable intersection of open, dense subsets of ${\mathcal X}(P)$. Because ${\mathcal X}(P)$ is a Baire space, any residual subset is dense. 

\begin{corollary}\label{rare} For a closed Riemannian manifold $P$, the set of vector fields on $P$ which do not possess nontrivial generalized symmetries is a residual subset of ${\mathcal X}(P)$.
\end{corollary}

\begin{proof} Let ${\mathcal V}$ be the set of $X$ in ${\mathcal X}(P)$ such that any equilibrium, if any, of the flow $\psi$ induced by $X$ is hyperbolic, and there is at least one periodic orbit for $\psi$. By Theorem \ref{noperiodic}, none of the vector fields in ${\mathcal V}$ possess nontrivial generalized symmetries. Let ${\mathcal H}$ be the subset of $X$ in ${\mathcal X}(P)$ such that the flow $\psi$ induced by $X$ has periodic orbits all of which are hyperbolic, and any equilibrium it has is hyperbolic. By definition, ${\mathcal H}\subset {\mathcal V}$, and by the Kupka-Smale Theorem, ${\mathcal H}$ is a residual subset of ${\mathcal X}(P)$.
\end{proof}

According to Theorem \ref{noperiodic}, flows which might possess nontrivial generalized symmetries are those that are equilibrium-free and without periodic orbits. The possession of a nontrivial generalized symmetry for an equilibrium-free flow places dynamical restrictions on the compact invariant sets of the flow. An invariant set $\Lambda$ for an equilibrium-free $\Phi\in{\rm Flow}(P)$ is {\it uniformly hyperbolic}\, if there is a continuous ${\bf T}\Phi$-invariant splitting ${\bf T}_p\Lambda = E^s(p) \oplus{\rm Span}(X_\Phi(p))\oplus E^u(p)$ for $p\in\Lambda$ and constants $C\geq 1$ and $\lambda\in(0,1)$ such that
\[ \Vert {\bf T}_p \Phi_t(v)\Vert \leq C\lambda^t \Vert v\Vert {\rm\ for\ }v\in E^s(p) {\rm\ and\ }t\geq 0\]
and
\[ \Vert {\bf T}_p \Phi_{-t}(v)\Vert \leq C\lambda^t \Vert v\Vert {\rm\ for\ }v\in E^u(p) {\rm\ and\ }t\geq 0.\]
Recall that if $\Lambda = P$ is uniformly hyperbolic for $\Phi$, then $\Phi$ is called Anosov.

\begin{corollary}\label{nonhyper} If $\Phi$ is an equilibrium-free flow on a closed Riemannian manifold $P$ which possesses a nontrivial generalized symmetry, then any compact invariant set for $\Phi$ is not uniformly hyperbolic $($and, in particular, $\Phi$ is not Anosov$)$.
\end{corollary}

\begin{proof} Suppose $\Lambda$ is a uniformly hyperbolic compact invariant set of $\Phi$. By the Anosov Closing Lemma, there exists a periodic orbit for $\Phi$. But with $\Phi$ being equilibrium-free and possessing a nontrivial generalized symmetry, Theorem \ref{noperiodic} implies that $\Phi$ does not have periodic orbits. Thus $\Lambda$ can not be uniformly hyperbolic.
\end{proof}

\section{Multipliers and Lyapunov Exponents}

The multiplier of a generalized symmetry for an equilibrium-free flow is related to the Lyapunov exponents of that generalized symmetry. For a closed Riemannian manifold $P$, the {\it Lyapunov exponent}\, of $R\in{\rm Diff}(P)$ at $p\in P$ is
\[ \chi_R (p,v) = \limsup_{m\to\infty} \frac{\log\Vert {\bf T}_p R^m(v)\Vert}{m}, \ v\in {\bf T}_p P\setminus\{0\}.\]
It is independent of the Riemannian norm on $P$ because $P$ is compact. As the following result shows, any trivial generalized symmetry for an equilibrium-free flow has zero Lyapunov exponents everywhere.

\begin{theorem}\label{Lyapunov} If $\psi$ is an equilibrium-free flow on a closed Riemannian manifold $P$, then for each $R\in S_\psi$, the one-dimensional distribution $E\subset {\bf T}P$ determined by $E(p) = {\rm Span}\big(X_\psi(p)\big)$ is $R$-invariant and
\[ \chi_R\big(p,X_\psi(p)\big) = \log\vert \rho_\psi(R)\vert {\rm \ for\ all\ }p\in P.\]
\end{theorem}

\begin{proof} Each $R\in S_\psi$ satisfies ${\bf T}_pR\big(X_\psi(R^{-1}(p)\big) = \rho_\psi(R) X_\psi(p)$ for all $p\in P$. Then ${\bf T}_pR(X_\psi(p)) = \rho_\psi(R)X_\psi(R(p))$. The distribution $E\subset {\bf T}P$ determined by $E(p) = {\rm Span}\big(X_\psi(p)\big)$ is $R$-invariant and one-dimensional because $\rho_\psi(R)\ne 0$ and $X_\psi(p)\ne 0$ for all $p\in P$. Furthermore, it follows for all $m\geq 1$ that 
\[ {\bf T}_pR^m(X_\psi(p)) = [\rho_\psi(R)]^m X_\psi\big(R^m(p)\big),\]
and so
\[ \Vert {\bf T}_pR^m(X_\psi(p))\Vert = \vert \rho_\psi(R)\vert^m \Vert X_\psi\big(R^m(p)\big)\Vert.\]
The flow $\psi$ being smooth and equilibrium-free on the compact manifold $P$ implies that $\Vert X_\psi\big(R^m(p)\big)\Vert$ is both bounded away from $0$ and bounded above uniformly in $m$ for each $p\in P$. Therefore
\[\chi_R\big(p,X_\psi(p)\big) = \limsup_{m\to\infty} \frac{ m \log\vert \rho_\psi(R)\vert + \log \Vert X_\psi\big(R^m(p)\big)\Vert}{m} = \log\vert \rho_\psi(R)\vert\]
for all $p\in P$.
\end{proof}

Global uniform hyperbolicity is a dynamical condition on a generalized symmetry of an equilibrium-free flow that guarantees that it is nontrivial. An $R\in{\rm Diff}(P)$ is {\it Anosov} if there is a continuous ${\bf T}R$-invariant splitting $T_p P = E^s(p)\oplus E^u(p)$ for $p\in P$, and constants $c>0$ and $\lambda\in(0,1)$ independent of $p\in P$, such that
\[ \Vert {\bf T}_p R^m(v)\Vert \leq c \lambda^m \Vert v\Vert {\rm \ for\ }v\in E^s(p){\rm\ and \ }m\geq 0,\]
and
\[ \Vert {\bf T}_p R^{-m}(v)\Vert \leq c \lambda^m \Vert v\Vert {\rm \ for\ }v\in E^u(p){\rm\ and \ }m\geq 0\]
with the angle between $E^s(p)$ and $E^u(p)$ bounded away from $0$ (see \cite{BP}).

\begin{theorem}\label{multiplierAnosov} Let $\psi$ be an equilibrium-free flow on a closed Riemannian manifold $P$. If $R\in S_\psi$ is Anosov, then $\vert\rho_\psi(R)\vert \ne 1$.
\end{theorem}

\begin{proof} Suppose $R\in S_\psi$ is Anosov and $\vert\rho_\psi(R)\vert=1$. If $\rho_\psi(R)=-1$, then replacing $R$ with $R^2$ gives $\rho_\psi(R) =1$ with $R$ Anosov. Let $E^s(p)\oplus E^u(p)$ be the continuous ${\bf T}R$-invariant splitting with its associated contraction estimates. The generating vector field for $\psi$ has a continuous decomposition $X_\psi(p) = v^s(p) + v^u(p)$ for $v^s(p)\in E^s(p)$ and $v^u(p)\in E^u(p)$. From the contraction estimates of ${\bf T}_pR$ on $E^s(p)$ and $E^u(p)$ it follows for all $p\in P$ that
\[ \Vert {\bf T}_pR^m(v^s(p))\Vert \to 0, \ \ \Vert {\bf T}_p R^{-m}(v^u(p))\Vert\to 0{\rm\ as\ } m\to\infty.\]
With $\rho_\psi(R) = 1$ and $X_\psi(p) = v^s(p) + v^u(p)$, the equation $R_*X_\psi = \rho_\psi(R)X_\psi$ becomes
\[ {\bf T}_p R^m (v^s(p) + v^u(p)) = X_\psi(R^m(p)) = v^s(R^m(p))+ v^u(R^m(p))\]
for  all $p\in P$ and $m\in {\mathbb Z}$. The ${\bf T}R$-invariance of $E^u(p)$ implies that
\[  {\bf T}_p R^m(v^u(p)) = v^u(R^m(p)).\]
For a fixed $p\in P$, there is by the compactness of $P$, a subsequence ${\rm R}^{m_i}(p)$ converging to a point, say $p_\infty$, as $m_i\to \infty$. Hence
\begin{align*}
X_\psi(p_\infty)
& = \lim_{i\to\infty} X_\psi(R^{m_i}(p)) \\
& = \lim_{i\to\infty} {\bf T}_pR^{m_i}(v^s(p)+v^u(p)) \\
& = \lim_{i\to\infty} \big[{\bf T}_p R^{m_i}(v^s(p)) + v^u(R^{m_i}(p))\big] \\
& = v^u(p_\infty).
\end{align*}
It now follows that
\[ \lim_{m\to\infty}X_\psi(R^{-m}(p_\infty)) = \lim_{m\to\infty}{\bf T}_{p_\infty} R^{-m}(X_\psi(p_\infty)) = \lim_{m\to\infty}{\bf T}_{p_\infty} R^{-m}(v^u(p_\infty))=0.\]
Compactness of $P$ implies that $X_\psi$ has a zero. But $\psi$ is an equilibrium-free flow, and therefore $\vert\rho_\psi(R)\vert\ne 1$.
\end{proof}

Any equilibrium-free flow with a generalized symmetry that is Anosov does not have any periodic orbits according to Lemma \ref{noperiodic} and Theorem \ref{multiplierAnosov}. On the other hand, for an equilibrium-free flow that is without nontrivial generalized symmetries, Theorem \ref{multiplierAnosov} implies that none of its generalized symmetries can be Anosov. However, the converse of Theorem \ref{multiplierAnosov} is false. As illustrated next, a partially hyperbolic diffeomorphism can be a nontrivial generalized symmetry of an equilibrium-free flow without periodic orbits.

\begin{example}\label{nonquasi}{\rm Let $P={\mathbb T}^n={\mathbb R}^n/{\mathbb Z}^n$, the $n$-torus, equipped with global coordinates $(\theta_1,\theta_2,\dots,\theta_n)$. The flow $\psi$ generated by
\[ X_\psi = \frac{\partial}{\partial\theta_1} + \frac{\partial}{\partial\theta_2} + \cdot\cdot\cdot + \frac{\partial}{\partial\theta_{n-2}} + \frac{\partial}{\partial\theta_{n-1}}  + \sqrt 2 \frac{\partial}{\partial\theta_n}\]
is equilibrium-free and without periodic orbits. Let $R\in{\rm Diff}({\mathbb T}^n)$ be induced by the ${\rm GL}(n,{\mathbb Z})$ matrix
\[ B = \begin{bmatrix} 1 & 0 & \hdots & 0 & 0 & 1 \\ 0 & 1 & \hdots & 0 & 0 & 1 \\ \vdots & \vdots & \ddots & \vdots & \vdots & \vdots \\ 0 & 0 & \hdots & 1 & 0 & 1 \\ 0 & 0 & \hdots & 0 &  1 & 1 \\ 0 & 0  & \hdots & 0 &  2 & 1\end{bmatrix},\]
that is, ${\bf T}R = B$. This $R$ is a nontrivial generalized symmetry of $\psi$ because $R_*X_\psi = (1+\sqrt 2)X_\psi$, i.e., $\rho_\psi(R) = 1+\sqrt 2$. However $B$ has an eigenvalue of $1$ of multiplicity $n-2$, and the eigenspace corresponding to this eigenvalue is an $(n-2)$-dimensional center distribution for $R$. The other two eigenvalues of $B$ are $1\pm\sqrt 2$ and the corresponding eigenspaces are the $1$-dimensional unstable and stable distributions for $R$ respectively. Thus $R$ is partially hyperbolic.
}\end{example}

\section{Restrictions on Multipliers}

Algebraic restrictions may occur on $\rho_\Phi(R)$ for a nontrivial generalized symmetry $R$ of an equilibrium-free flow $\Phi$ on a compact manifold $P$ without boundary. This happens when there are interactions beyond $R_*X_\Phi = \rho_\Phi(R)X_\Phi$ between the dynamics of $R$ and $\Phi$ on submanifolds diffeomorphic to ${\mathbb T}^k$ for some $2\leq k\leq {\rm dim}(P)$. A {\it real algebraic integer}\, is the root of a monic polynomial with integer coefficients, and its degree is the degree of its minimal polynomial. As shown by Wilson in \cite{WI}, there is for any integer $k$ with $2\leq k\leq {\rm dim}(P)-2$, an equilibrium-free flow $\Phi$ on $P$ which has an invariant submanifold $N$ diffeomorphic to ${\mathbb T}^k$ on which $\Phi$ is {\it minimal}, i.e., every orbit of $\Phi$ in $N$ is dense in $N$. The well-known prototype of a minimal flow on ${\mathbb T}^k$ is an {\it irrational flow}, i.e., a $\psi$ on ${\mathbb T}^k$ for which $X_\psi$ is a constant vector field whose components (or frequencies) are rationally independent (see \cite{FA}), i.e., linearly independent over ${\mathbb Q}$.  According to Basener in \cite{BA}, any minimal flow on ${\mathbb T}^n$ is topologically conjugate to an irrational flow.

\begin{lemma}\label{subset} If a flow $\phi$ on a compact manifold $N$ without boundary is topologically conjugate to a flow $\psi$ on ${\mathbb T}^k$ $(k={\rm dim}(N))$ and $\psi$ is minimal, then $M_\phi\subset M_\psi$ and each $\mu\in M_\phi\setminus\{1,-1\}$ is an algebraic integer of degree between $2$ and $k$ inclusively.
\end{lemma}

\begin{proof} Suppose there is a homeomorphism $h:N\to{\mathbb T}^k$ such that $h\phi_t = \psi_t h$ for all $t\in{\mathbb R}$. Without loss of generality, it is assumed that $\psi$ is an irrational flow, since any minimal flow on ${\mathbb T}^k$ is topologically conjugate to an irrational flow.

Let $V\in S_\phi$ and set $\mu=\rho_\phi(V)$. In terms of the homeomorphism $Q=hVh^{-1}$ on ${\mathbb T}^k$, the multiplier $\mu$ passes through $h$ to $\psi$:
\[ Q\psi(t,\theta) = hR\phi(t,h^{-1}(\theta)) = h\phi(\mu t,Rh^{-1}(\theta)) = \psi(\mu t, Q(\theta)), {\rm\ for\ all\ }t\in{\mathbb R}, \theta \in{\mathbb T}^k,\] 
i.e., $Q\psi_t = \psi_{\mu t}Q$. This does not say yet that $Q$ is a generalized symmetry of $\psi$ with multiplier $\mu$ because $Q$ is only a homeomorphism at the moment.

Following \cite{AP}, the homeomorphism $Q$ of ${\mathbb T}^k$ lifts uniquely to $\hat Q(x)=\hat L(x) + \hat U(x) +\hat c$ on ${\mathbb R}^k$, i.e., $\pi \hat Q = Q\pi$ where $\pi:{\mathbb R}^k\to{\mathbb T}^k$ is the covering map. Here the linear part of this lift is $\hat L(x) = Bx$ for $B\in {\rm GL}(k,{\mathbb Z})$; the periodic part is $\hat U(x)$, i.e., $\hat U(x+\nu) = \hat U(x)$ for all $x\in{\mathbb R}^k$ and all $\nu\in {\mathbb Z}^k$, is continuous and satisfies $\hat U(0)=0$; and the constant part is $\hat c\in[0,1)^k$. A lift of the irrational flow $\psi$ to ${\mathbb R}^n$ is $\hat \psi(t,x) = x+td$ where $d=X_\psi$ and $x\in{\mathbb R}^k$. For all $t\in{\mathbb R}$, a lift of $Q\psi_t$ to ${\mathbb R}^k$ is $\hat Q\hat \psi_t$ and a lift of $\psi_{\mu t}Q$ is $\hat \psi_{\mu t}\hat Q$. These two lifts differ by a constant $m\in{\mathbb Z}^k$ since $Q\psi_t = \psi_{\mu t}Q$ for all $t\in{\mathbb R}$:
\[ \hat Q\hat \psi(t,x)  = \hat \psi(\mu t,\hat Q(x)) + m {\rm\ for\ all\ }t\in{\mathbb R}, x\in{\mathbb R}^k.\]
Since
\[ \hat \psi(t,x) = Bx + tBd + \hat U(x+td) + \hat c\]
and 
\[ \hat\psi(\mu t,\hat Q(x)) = Bx + \hat U(x) + \hat c + \mu t d,\]
it follows that
\[ \hat U(x+td) - \hat U(x) = - t(B-\mu I)d + m {\rm\ for\ all\ }t\in{\mathbb R}, x\in{\mathbb R}^k,\]
where $I$ is the identity matrix. Evaluation of this at $x=0$ gives
\[ \hat U(td) = -t(B-\mu I)d + m {\rm\ for\ all\ }t\in{\mathbb R}.\]
However, $\hat U$ is bounded since it is continuous and periodic. This boundedness implies that $(B-\mu I)d = 0$, and so $\hat U(td) = m$ for all $t\in{\mathbb R}$. Evaluation of this at $t=0$ shows that $m=0$ because $\hat U(0)=0$. Thus
\[ 0=\hat U(td) = \hat U (\hat \psi_t(0)){\rm \ for\  all\ }t\in{\mathbb R}.\] Since $\hat U$ is periodic and continuous on ${\mathbb R}^n$, it is a lift of a homeomorphism $U$ on ${\mathbb T}^n$. A lift of $U\psi_t$ is $\hat U\hat\psi_t$, and so
\[ 0=\pi \hat U(\hat \psi_t(0)) = U(\psi_t(0)) {\rm\ for\ all\ }t\in{\mathbb R}.\]
The minimality of $\psi$ implies that $U$ is $0$ on a dense subset of ${\mathbb T}^n$. By continuity, $U = 0$ which implies that $\hat U=0$. Thus $\hat Q$ is $C^\infty$, and so $Q$ is a diffeomorphism, whence $Q\in S_\psi$ with $\rho_\psi(Q) = \mu$. Since $\mu\in M_\phi$, then $M_\phi\subset M_\psi$.

The multiplier $\mu \in M_\psi$ is a real algebraic integer of degree at most $k$ because it satisfies $(B-\mu I)d = 0$ for nonzero $d$, i.e., the characteristic polynomial of $B$ is a monic polynomial of degree $k$ with integer coefficients. The only rational roots this characteristic polynomial can have are $\pm 1$ since $\psi$ is an irrational flow, i.e., $M_\psi\cap{\mathbb Q}=\{1,-1\}$ (see Corollary 4.4 in \cite{BA2}). However, if $\mu \ne \pm 1$, then the minimal polynomial for $\mu$ has degree between $2$ and $k$ inclusively.
\end{proof}

\begin{theorem}\label{algebraicmultiplier} Suppose for an equilibrium-free flow $\Phi$ on $P$ there is a $\Phi$-invariant compact submanifold $N$ without boundary and $R\in S_\Phi$ with $\vert\rho_\Phi(R)\vert \ne 1$ such that $R(N)\cap N\ne \emptyset$. If ${\rm dim}(N)=2$ and $N$ is orientable, then $\rho_\Phi(R)$ is a real algebraic integer of degree $2$. If ${\rm dim}(N)\geq 3$ with $N$ diffeomorphic to ${\mathbb T}^{{\rm dim}(N)}$ and $\Phi\vert N$ is a minimal flow, then $\rho_\Phi(R)$ is a real algebraic integer of degree between $2$ and ${\rm dim}(N)$ inclusively.
\end{theorem}

\begin{proof} Let $\mu=\rho_\Phi(R)$ with $\vert\mu\vert\ne 1$ and $k={\rm dim}(N)$. By Theorem \ref{noperiodic} there are no periodic orbits for the equilibrium-free flow $\Phi\vert N$. If $k=2$ and $N$ is orientable, the Poincar\'e-Bendixson Theorem implies that $N$ is diffeomorphic to ${\mathbb T}^2$, and that $\Phi\vert N$ is minimal. If $k\geq 3$, it is assumed that $N$ is diffeomorphic to ${\mathbb T}^k$ and that $\Phi\vert N$ is minimal. 

The submanifold $R(N)$ is $\Phi$-invariant because $N$ is $\Phi$-invariant and $R\in S_\Phi$, i.e., for $p\in R(N)$ and $q= R^{-1}(p)\in N$,
\[ \Phi(t,p) = \Phi(t,R(q)) = R\Phi(t/\mu,q) \subset R(N) {\rm \ for\ all\ }t\in{\mathbb R}.\]
By hypothesis, there is $\tilde p\in R(N)\cap N$. By the $\Phi$-invariance of $N$ and $R(N)$ and the minimality of $\Phi\vert N$ it follows that $\overline{{\mathcal O}_\Phi(\tilde p)} = N$ and $\overline{{\mathcal O}_\Phi(\tilde p)} = R(N)$. This gives $R(N) = N$, i.e., that $N$ is $R$-invariant.

The nontrivial generalized symmetry $R$ restricts to a nontrivial generalized symmetry of $\Phi\vert N$ because $N$ is $\Phi$-invariant and $R$-invariant. If $V=R\vert N$ and $\phi$ is the flow on $N$ determined by $X_\phi = X_\Phi\vert N$, then $R_*X_\Phi = \mu X_\Phi$ becomes 
\[ {\bf T} V X_\phi (p) = \mu X_\phi (V(p)){\rm\ for\ }p\in N.\]
Since $V\in{\rm Diff}(N)$, then $V\in S_\phi$ with $\rho_\phi(V) = \mu$.

By \cite{BA}, minimality of $\phi=\Phi\vert N$ with $N$ diffeomorphic to ${\mathbb T}^n$ implies that $\phi$ is topologically conjugate to an irrational flow $\psi$. Applying Lemma \ref{subset} shows that $\mu$ is an algebraic integer of degree between $2$ and $k$ inclusively.
\end{proof}

\begin{corollary}\label{algebraicnature} Suppose $\Phi$ is a minimal flow on $\mathbb T^n$, and $R\in S_\Phi$. If $\vert \rho_\Phi(R)\vert \ne 1$, then $\rho_\Phi(R)$ is an algebraic integer of degree between $2$ and $n$ inclusively.
\end{corollary}

\begin{proof} Any nontrivial generalized symmetry $R$ of $\Phi$ together with $N=\mathbb T^n$ satisfy the conditions of Theorem \ref{algebraicmultiplier}.
\end{proof}

\section{A Group-Theoretical Characterization of Irrational Flows}

Minimality of an equilibrium-free flow on ${\mathbb T}^n$ places a semidirect product structure the generalized symmetry group of that flow. A group ${\mathcal S}$ is the {\it semidirect product}\, of two subgroups ${\mathcal N}$ and ${\mathcal H}$ if ${\mathcal N}$ is a normal subgroup of ${\mathcal S}$, if ${\mathcal S}={\mathcal N}{\mathcal H}$, and if ${\mathcal N}\cap {\mathcal H}$ is the identity element of ${\mathcal S}$. Notational this is written
\[ {\mathcal S} = {\mathcal N}\rtimes_\Gamma {\mathcal H},\]
where $\Gamma: {\mathcal H}\to {\rm Aut}({\mathcal N})$ is the conjugating homomorphism of the semidirect product, i.e., $\Gamma({\mathfrak h})({\mathfrak n}) = {\mathfrak h}{\mathfrak n}{\mathfrak h}^{-1}$ for ${\mathfrak h}\in {\mathcal H}$ and ${\mathfrak n}\in {\mathcal N}$. A normal subgroup of $S_\psi$ is $\ker \rho_\psi$ for any flow $\psi$. A normal subgroup of ${\rm Diff}(\mathbb T^n)$ is the abelian group ${\rm Trans}(\mathbb T^n)$ of translations. Each {\it translation}\, on ${\mathbb T}^n$ is of the form ${\mathcal T}_c(\theta) = \theta + c$ for $c\in{\mathbb T}^n$. If $\psi$ is a flow on ${\mathbb T}^n$ with $X_\psi$ a constant, then ${\rm Trans}({\mathbb T}^n)\subset {\rm ker}\rho_\psi$ because $({\mathcal T}_c)_* X_\psi = X_\psi$ for all $c\in{\mathbb T}^n$.

\begin{lemma}\label{kernel} Let $\psi$ be a flow on ${\mathbb T}^n$ for which $X_\psi$ is a constant. If ${\rm ker}\rho_\psi = {\rm Trans}({\mathbb T}^n)$, then $\psi$ is irrational.
\end{lemma}

\begin{proof} Suppose that $\psi$ is not irrational. Then the components of $X_\psi$ are not rationally independent. Up to a permutation of the coordinates $\theta_1,\dots,\theta_n$ on ${\mathbb T}^n$, i.e., a smooth conjugacy, it can be assumed the first $l$ components of $X_\psi$ are the smallest subset of the components that are linearly dependent  over ${\mathbb Q}$. Specifically, writing $X_\psi = [a_1,a_2,\dots,a_n]^T$, there is a least integer $l$ with $1\leq l\leq n$ such that
\[ k_1a_1+\cdot\cdot\cdot + k_la_l =0\]
with $k_i\in{\mathbb Z}\setminus\{0\}$ for all $i=1,\dots,l$. The existence of an $R\in{\rm ker}\rho_\psi\setminus{\rm Trans}({\mathbb T}^n)$ will be exhibited separately in the cases of $l<n$ and $l=n$.

Case $1\leq l\leq n-1$. The $R\in{\rm Diff}({\mathbb T}^n)$ induced by the ${\rm GL}(n,{\mathbb Z})$ matrix
\[ B  = \begin{bmatrix} 1 & 0 & \dots & 0 & 0 & 0 & \dots & 0 & 0 \\ 0 & 1 & \dots & 0 & 0 & 0 & \dots & 0 & 0 \\ \vdots & \vdots & \ddots & \vdots & \vdots & \vdots & \ddots & \vdots & \vdots \\ 0 & 0 & \dots & 1 & 0 & 0 & \dots & 0 & 0 \\ 0 & 0 & \dots & 0 & 1 & 0 & \dots & 0 & 0 \\ 0 & 0 & \dots & 0 & 0 & 1 & \dots & 0 & 0 \\ \vdots & \vdots & \ddots & \vdots & \vdots & \vdots & \ddots & \vdots & \vdots \\ 0 & 0 & \dots & 0 & 0 & 0 & \dots & 1 & 0 \\  k_1 & k_2 & \dots & k_{l-1} & k_l & 0 & \dots & 0 & 1\end{bmatrix}\]
satisfies $R_*X_\psi = X_\psi$. So $R\in {\rm ker}\rho_\psi$, but $R\not\in{\rm Trans}({\mathbb T}^n)$.

Case $l=n$. Here none of the $k_1,k_2,\dots,k_n$ are zero, and the integer-entry matrix
\[ B = \begin{bmatrix} 1 & 0 & \dots & 0 & 0 \\ 0 & 1 & \dots & 0 & 0 \\ \vdots & \vdots & \ddots & \vdots & \vdots \\ 0 & 0 & \dots & 1 & 0 \\ k_1 & k_2 & \dots & k_{n-1} & k_n\end{bmatrix}\]
induces a $k_n$-to-$1$ smooth surjection $g$ of ${\mathbb T}^n$ to itself. The flow $\phi$ on ${\mathbb T}^n$ generated by $X_\phi = [a_1,\dots,a_{n-1},0]^T$ is a factor of $\psi$ i.e., $g\psi_t = \phi_t g$ for all $t\in\mathbb R$, because $BX_\psi = X_\phi$.

The components $a_1,\dots,a_{n-1}$ of $X_\phi$ are rationally independent. For, if they were not, there would then be integers $k_1^\prime,\dots,k_{n-1}^\prime$, not all zero, such that $k_1^\prime a_1 + \cdot\cdot\cdot + k_{n-1}^\prime a_{n-1}=0$, which would contradict the minimality of $l=n$. With $e_1,\dots,e_n$ as the standard basis for ${\mathbb R}^n$ and $E={\rm Span}(e_1,\dots,e_{n-1})$ an $(n-1)$-dimensional subspace of ${\mathbb R}^n$, the closure of ${\mathcal O}_\phi(0)$ is $\pi(E)$. Thus $\overline{{\mathcal O}_\phi(0)}$ is an embedded submanifold of ${\mathbb T}^n$ diffeomorphic to ${\mathbb T}^{n-1}$.

The semiconjugacy $g$ between $\psi$ and $\phi$ along with $g(0)=0$ imply that $\overline{{\mathcal O}_\psi(0)}$ is an embedded submanifold of ${\mathbb T}^n$ that is diffeomorphic to ${\mathbb T}^{n-1}$. In particular, the $(n-1)$-dimensional subspace $U$ of ${\mathbb R}^n$ determined by $B(U) = E$ satisfies $\pi(U) = \overline{{\mathcal O}_\psi(0)}$, where $U={\rm Span}(u_1,u_2,\dots,u_{n-1})$ for $B(u_i) = e_i$, $i=1,\dots,n-1$. Furthermore, 
\[ X_\psi = a_1u_1 + \cdot\cdot\cdot + a_{n-1}u_{n-1} \in U\]
because $k_1a_1+\cdot\cdot\cdot + k_na_n = 0$.

For $\epsilon>0$, $V$ a nonempty subset of ${\mathbb R}^n$, and $m\in{\mathbb Z}^n$, define $N_\epsilon(V)$ to be the set of points in ${\mathbb R}^n$ less than a distance of $\epsilon$ from $V$, and define $V+m$ to be the translation of $V$ by $m$. By the definition of $E$, if $E+m\ne E$, then $N_{1/2}(E+m)\cap N_{1/2}(E)=\emptyset$. Since $B(U)=E$ and $B({\mathbb Z}^n)\subset {\mathbb Z}^n$, there exists $\epsilon>0$ such that if $U+m\ne U$, then $N_\epsilon(U+m)\cap N_\epsilon(U)=\emptyset$.

The vector $e_n\not \in U$. Let $x_1,\dots, x_n$ be the coordinates on ${\mathbb R}^n$ that correspond to the basis $u_1,\dots,u_{n-1},e_n$. Let $f:{\mathbb R}\to{\mathbb R}$ be a smooth bump function with $f(0)=1$ and whose support has length smaller than $\epsilon$. In terms of the coordinates $x_1,\dots, x_n$, define a smooth vector field on $N_\epsilon(U)$ by
\[ Y = f(x_n)\frac{\partial}{\partial x_n}.\]
Extend this vector field to all of ${\mathbb R}^n$ by translation to $N_\epsilon(U+m)$ for those $m\in{\mathbb Z}^m$ for which $U+m\ne U$, and to the remainder of ${\mathbb R}^n$ as $0$. The extended vector field is globally Lipschitz, and so determines a flow $\xi$ on ${\mathbb R}^n$.

Since the vector field generating $\xi$ is invariant under translations by $m\in{\mathbb Z}$, the time-one map $\xi_1$ is also invariant under these translations. Thus $\xi_1$ is a lift of an $R\in{\rm Diff}({\mathbb T}^n)$. In terms of the coordinates $x_1,\dots, x_n$, the derivative of $\xi_1$ at any point $x\in{\mathbb R}^n$ is of the form
\[ {\bf T}_x\xi_1 = \begin{bmatrix} 1 & 0 & \dots & 0 & 0 \\ 0 & 1 & \dots & 0 & 0 \\ \vdots & \vdots & \ddots & \vdots & \vdots \\ 0 & 0 & \dots & 1 & 0 \\ 0 & 0 & \dots & 0  & \ast \end{bmatrix},\]
and so ${\bf T}_x\xi _1(u) = u$ for $u\in U$. Since $X_\psi\in U$, this means that ${\bf T}_x \xi_1(X_\psi) =  X_\psi$. Since $X_\psi$ is a constant vector field, then $(\xi_1)_* X_\psi = X_\psi$. Since $\xi_1$ is a lift of $R$, it now follows that $R_*X_\psi = X_\psi$. Therefore $R\in {\rm ker}\rho_\psi$ but $R\not\in{\rm Trans}({\mathbb T}^n)$.
\end{proof}

The group ${\rm Aut}({\mathbb T}^n)$ of automorphisms of $\mathbb T^n$ is naturally identified with ${\rm GL}(n,{\mathbb Z})$. For $\mathcal T_c\in {\rm Trans}(\mathbb T^n)$ and $B\in {\rm GL}(n,\mathbb Z)$, the composition of $\mathcal T_c$ with $B$ is written $\mathcal T_c  B= B+c$. 

\begin{theorem}\label{characterization} Let $\psi$ be a flow on ${\mathbb T}^n$ with $X_\psi$ a nonzero constant vector. Then $\psi$ is irrational if and only if there exists a subgroup $H$ of ${\rm GL}(n,{\mathbb Z})$ isomorphic to $M_\psi$ such that $S_\psi = {\rm Trans}({\mathbb T}^n) \rtimes_\Gamma H$.
\end{theorem}

\begin{proof} Suppose that $\psi$ is irrational. This implies (by Theorem 5.5 in \cite{BA2}) that
\[ S_\psi = {\rm ker}\rho_\psi \rtimes_\Gamma H_\psi,\]
where $H$ is a subgroup of ${\rm GL}(n,{\mathbb Z})$ isomorphic to $M_\psi$. Furthermore, the irrationality of $\psi$ implies (by Corollary 4.7 in \cite{BA2}) that ${\rm ker}\rho_\psi={\rm Trans}({\mathbb T}^n)$.

Now suppose that $\psi$ is not irrational. By Lemma \ref{kernel}, there is $R\in {\rm ker}\rho_\psi\setminus{\rm Trans}({\mathbb T}^n)$. If there were a subgroup $H$ of ${\rm GL}(n,{\mathbb Z})$ isomorphic to $M_\psi$ such that $S_\psi = {\rm Trans}(\mathbb T^n)\rtimes_\Gamma H$, then $R=\mathcal T_c B$ for some $c\in{\mathbb T}^n$ and $B\in H$. Hence $1=\rho_\psi(R) = \rho_\psi(\mathcal T_c)\rho_\psi(B)= \rho_\psi(B)$. However, since $H$ is isomorphic to $M_\psi$, there is only one element of $H$ which corresponds to the multiplicative identity $1$ of $\mathbb R^\times$, and that is $I$, the identity matrix. This means that $B=I$, and so $R=\mathcal T_c$, a contradiction.
\end{proof}

\section{Global Rigidity for Certain ${\mathbb Z}^d$ Anosov Actions}

Global rigidity is about when a ${\mathbb Z}^d$ Anosov action, which as is well-known is topologically conjugate to an algebraic ${\mathbb Z}^d$ action (see \cite{KK}), is smoothly conjugate to an algebraic ${\mathbb Z}^d$ action. A ${\mathbb Z}^d$ action on ${\mathbb T}^n$ is a monomorphism $\alpha:{\mathbb Z}^d\to{\rm Diff}({\mathbb T}^n)$. It is Anosov if there is $m\in{\mathbb Z}^d\setminus\{0\}$ with $\alpha(m)$ Anosov, is {\it algebraic}\, if $\alpha({\mathbb Z}^d)\subset {\rm GL}(n,{\mathbb Z})$, and more generally is {\it affine}\, if $\alpha({\mathbb Z}^d)\subset{\rm Trans}({\mathbb T}^n)\rtimes_\Gamma{\rm GL}(n,{\mathbb Z})$. Algebraic ${\mathbb Z}^d$-actions are found in algebraic number theory (see \cite{KKS} and \cite{Sc}). 

\vskip0.2cm
\noindent{\it Proof of Theorem \ref{Anosovalgebraic}.} Although well-known, the argument for the existence of a common fixed point of a finite index subgroup of a ${\mathbb Z}^d$ Anosov action is included for completeness. Let $m_0\in{\mathbb Z}^d$ be such that $\alpha(m_0)$ is Anosov. Then $\alpha(m_0)$ is topologically conjugate to a hyperbolic automorphism of ${\mathbb T}^n$ (see \cite{Ma}), and so $\alpha(m_0)$ has a finite number of fixed points, $f_1,\dots,f_l$. Let $\alpha(m_1),\dots,\alpha(m_d)$ be a generating set for $\alpha({\mathbb Z}^d)$. Since $\alpha({\mathbb Z}^d)$ is abelian, then for all $i=1,\dots,d$ and $j=1,\dots,l$,
\[ \alpha(m_0)\alpha(m_i)(f_j) = \alpha(m_i)\alpha(m_0)(f_j) = \alpha(m_i)(f_j).\]
This means that $\alpha(m_i)(f_j)$ is one of the finitely many fixed points of $\alpha(m_0)$. Since each $\alpha(m_i)$ is invertible, there is a positive finite integer $r_i$ such that $\alpha(m_i)^{r_i}(f_1) = f_1$. Thus the finite index subgroup of $\alpha({\mathbb Z}^d)$ generated by $\alpha(m_1)^{r_1}, \dots,\alpha(m_d)^{r_d}$ has $f_1$ as a common fixed point.

By hypothesis, there is a $g\in{\rm Diff}({\mathbb T}^n)$ and an irrational flow $\phi$ on ${\mathbb T}^n$ for which $X_\phi = g_* X_\Phi$. By Theorem \ref{characterization}, there is a subgroup $H_\phi$ of ${\rm GL}(n,\mathbb Z)$ isomorphic to $M_\phi$ such that $S_\phi={\rm Trans}({\mathbb T}^n)\rtimes_\Gamma H_\phi$. For $h\in{\rm Diff}({\mathbb T}^n)$ given by $h={\mathcal T}_{-g(f_1)}\circ g$ it follows that $h_*X_\Phi = X_\phi$ because $({\mathcal T}_c)_*X_\phi = X_\phi$ for any $c\in{\mathbb T}^n$. Then $\Phi$ and $\phi$ are projectively conjugate, and so, as mentioned in Section 2, $\Delta_h(S_\phi) = S_\Phi$. The inclusion $\alpha({\mathbb Z}^d)\subset S_\Phi$ implies that
\[ \Delta_{h^{-1}}(\alpha({\mathbb Z}^d))\subset {\rm Trans}({\mathbb T}^n)\rtimes_\Gamma H_\phi.\]
This means that $\alpha$ is $C^\infty$-conjugate to an affine ${\mathbb Z}^d$-action.

For each $\alpha(m)$ in the finite index subgroup generated by $\alpha(m_1)^{r_1},\dots,\alpha(m_d)^{r_d}$ there are $B_m\in H_\phi$ and $c_m\in {\mathbb T}^n$ such that $\Delta_{h^{-1}}(\alpha(m)) = B_m + c_m$. Since $h(f_1) = 0$, then
\[ c_m=(B_m+c_m)(0) = h\circ\alpha(m)\circ h^{-1}(0) = h \circ \alpha(m)(f_1) = h(f_1) = 0.\]
This means that $\Delta_{h^{-1}}(\alpha(m))\in H_\phi$. Hence the finite index subgroup of $\alpha({\mathbb Z}^d)$ generated by $\alpha(m_1)^{r_1},\dots,\alpha(m_d)^{r_d}$, which is isomorphic to ${\mathbb Z}^d$, is $C^\infty$-conjugate to an algebraic ${\mathbb Z}^d$-action. Since $H_\phi$ contains a ${\mathbb Z}^d$ subgroup, and $M_\phi$ is isomorphic to $H_\phi$, then $M_\phi$ contains a ${\mathbb Z}^d$ subgroup. By absolute invariance of the multiplier group under projective conjugacy, $M_\Phi$ contains a ${\mathbb Z}^d$ subgroup.  \hfill $\Box$

\vskip0.2cm
Any quasiperiodic flow on ${\mathbb T}^n$ whose generalized symmetry group contains a ${\mathbb Z}^d$ Anosov action must possess nontrivial generalized symmetries since its multiplier group contains a ${\mathbb Z}^d$ subgroup by Theorem \ref{Anosovalgebraic} (cf.\,Theorem \ref{multiplierAnosov}). This puts a necessary condition on the quasiperiodic flows to which Theorem \ref{Anosovalgebraic} does apply. The quasiperiodic flows of Koch type mentioned in the next section satisfy this necessary condition. However, there are quasiperiodic flows that do not, as is illustrated next.

\begin{example}\label{trivialquasi}{\rm On ${\mathbb T}^n$, let $\psi$ be the flow generated by
\[ X_\psi = \frac{\partial}{\partial\theta_1} + \pi\frac{\partial}{\partial\theta_2} + \cdot\cdot\cdot + \pi^{n-1} \frac{\partial}{\partial\theta_n}.\]
This flow is quasiperiodic, since if its frequencies $1,\pi,\dots,\pi^{n-1}$ were linearly dependent over ${\mathbb Q}$ then $\pi$ would be algebraic. Quasiperiodicity of $\psi$ implies that each $\mu\in M_\psi$ is a real algebraic integer of degree at most $n$, and moreover, $M_\psi\cap{\mathbb Q}=\{1,-1\}$ (see Corollary 4.4 in \cite{BA2}). Suppose $\mu\in M_\psi\setminus\{1,-1\}$. Then there is $R\in S_\psi$ such that $R_*X_\psi = \mu X_\psi$. Quasiperiodicity of $\psi$ implies that ${\bf T}R=B\in{\rm GL}(n,{\mathbb Z})$ (see Theorem 4.3 in \cite{BA2}). Then $R_*X_\psi = \mu X_\psi$ becomes $BX_\psi = \mu X_\psi$, and for $B=(b_{ij})$, it follows that
\[ \mu = b_{11} + b_{12}\pi + \cdot\cdot\cdot + b_{1n}\pi^{n-1}.\]
If $b_{12}=\cdot\cdot\cdot = b_{1n} = 0$, then $\mu = b_{11}\in{\mathbb Z}$. But $M_\psi\cap{\mathbb Q}=\{1,-1\}$, and so $\mu = \pm 1$. This contradiction means that one of $b_{12},\dots,b_{1n}$ is nonzero. Because each multiplier of $\psi$ is a real algebraic integer of degree at most $n$, there is a monic polynomial $l(z)$ in the polynomial ring ${\mathbb Z}[z]$ such that $l(\mu)=0$. But this implies that $\pi$ is a root of a polynomial in ${\mathbb Z}[z]$, making $\pi$ algebraic. This shows that $M_\psi=\{1,-1\}$, and so $\psi$ does not possess nontrivial generalized symmetries.
}\end{example}

\section{Classification of Certain Equilibrium-Free Flows}

Quasiperiodic flows of Koch type are algebraic in nature and provide foliations which are often preserved by a topologically irreducible ${\mathbb Z}^d$ Anosov action (see \cite{BA5}, \cite{BA6}, and \cite{KKS} for such examples). A flow on ${\mathbb T}^n$ is {\it quasiperiodic of Koch type}\, if it is projectively conjugate to an irrational flow whose frequencies form a ${\mathbb Q}$-basis for a real algebraic number field ${\mathbb F}$ of degree $n$ over ${\mathbb Q}$. For a quasiperiodic flow $\Phi$ of Koch type, the real algebraic number field ${\mathbb F}$ of degree $n$ associated to it is unique, and its multiplier group is a finite index subgroup of the group of units ${\mathfrak o}_{\mathbb F}^\times$ in the ring of integers of ${\mathbb F}$ (see Theorem 3.3 in \cite{BA5}).  By Dirichlet's Unit Theorem (see \cite{SD}), there is $d\geq 1$ such that ${\mathfrak o}_{\mathbb F}^\times$ is isomorphic to ${\mathbb Z}_2\oplus{\mathbb Z}^d$, and so every quasiperiodic flow of Koch type always possesses nontrivial generalized symmetries.

Topological irreducibility of a ${\mathbb Z}^d$ action $\alpha$ is a condition on the topological factors that $\alpha$ has. A ${\mathbb Z}^d$ action $\alpha^\prime$ on ${\mathbb T}^{n^\prime}$ is a {\it topological factor}\, of $\alpha$ if there is a continuous surjection $h:{\mathbb T}^n\to{\mathbb T}^{n^\prime}$ such that $h\circ \alpha = \alpha^\prime\circ h$. A topological factor $\alpha^\prime$ of $\alpha$ is {\it finite}\, if the continuous surjection $h$ is finite-to-one everywhere. A ${\mathbb Z}^d$ action $\alpha$ is {\it topologically irreducible}\, if every topological factor $\alpha^\prime$ of $\alpha$ is finite.

For an algebraic ${\mathbb Z}^d$ action $\alpha$, there is a stronger sense of irreducibility, one that uses the group structure of ${\mathbb T}^n$. An algebraic ${\mathbb Z}^d$ action $\alpha^\prime$ on ${\mathbb T}^{n^\prime}$ is an {\it algebraic factor}\, of $\alpha$ if there is a continuous homomorphism $h:{\mathbb T}^n\to{\mathbb T}^{n^\prime}$ such that $h\circ \alpha = \alpha^\prime\circ h$. An algebraic factor $\alpha^\prime$ of $\alpha$ is {\it finite} if the continuous homomorphism $h$ is finite-to-one everywhere. An algebraic ${\mathbb Z}^d$ action $\alpha$ is {\it algebraically irreducible}\, if every algebraic factor $\alpha^\prime$ of $\alpha$ is finite. Algebraic irreducibility of a higher rank algebraic $\mathbb Z^d$ action $\alpha$ is equivalent to there being an $m\in \mathbb Z^d$ such that $\alpha(m)$ has an irreducible characteristic polynomial (see Proposition 3.1 on p.\,726 in \cite{KKS}; cf.\,\cite{BE}).

\vskip0.2cm
\noindent{\it Proof of Theorem \ref{flowKoch}.} Identify ${\bf T}{\mathbb T}^n$ with ${\mathbb T}^n\times{\mathbb R}^n$, and place on the fiber the standard Euclidean norm $\Vert\cdot\Vert$. By the hypotheses, there is $h\in{\rm Diff}({\mathbb T}^n)$ and a hyperbolic $B\in{\rm GL}(n,{\mathbb Z})$ such that $\Delta_{h^{-1}}(\alpha(m_0)) = B$. Every point of ${\mathbb T}^n$ is Lyapunov regular for $B$. The Oseledets decomposition associated with $\chi_B$ is
\[ {\bf T}_\theta {\mathbb T}^n = \bigoplus_{i=1}^k E^i_B,\]
where $E_B^i$, $i=1,\dots,k$, are the invariant subspaces of $B$ which are independent of $\theta$. Since $\Delta_{h^{-1}}(\alpha(m_0))=B$, every point of ${\mathbb T}^n$ is Lyapunov regular for $\alpha(m_0)$. Set
\[ E^i_{\alpha(m_0)}(\theta) = {\bf T}_{h(\theta)}h^{-1}\big( E^i_B), \ \ \theta\in{\mathbb T}^n.\]
The Oseledets decomposition associated with $\chi_{\alpha(m_0)}$ is then
\[ {\bf T}_\theta {\mathbb T}^n = \bigoplus_{i=1}^k E^i_{\alpha(m_0)}(\theta).\]

For $\mu = \rho_\Phi(\alpha(m_0))$, the hypothesis that the multiplicity of the value $\log\vert \mu \vert$ of $\chi_{\alpha(m_0)}$ is one at a point $\hat \theta\in{\mathbb T}^n$ implies that there is $1\leq l\leq k$ such that ${\rm dim}\big(E^l_{\alpha(m_0)}(\hat\theta)\big) = 1$ and
\[ \chi_{\alpha(m_0)}(\hat\theta,v) = \log\vert \mu \vert {\rm\ for\ } v\in E^l_{\alpha(m_0)}(\hat\theta)\setminus\{0\}.\]
By the definition $E^l_{\alpha(m_0)}(\hat\theta) = {\bf T}_{h(\hat\theta)}h^{-1}\big( E^l_B)$, it follows that ${\rm dim}\big( E^l_B\big) = 1$. Furthermore, since $\Delta_{h^{-1}}(\alpha(m_0))=B$ and $\chi_B$ is independent of $\theta$, it follows that $\chi_B(\theta,v) = \log\vert\mu\vert$ for all $v\in E^l_B\setminus\{0\}$ and for all $\theta\in{\mathbb T}^n$. The definition $E^l_{\alpha(m_0)}(\theta) = {\bf T}_{h(\theta)}h^{-1}\big( E^l_B)$ and the independence of $E^l_B$ from $\theta$ imply for all $\theta\in{\mathbb T}^n$ that ${\rm dim} \big(E^l_{\alpha(m_0)}(\theta)\big) = 1$ and $\chi_{\alpha(m_0)}(\theta,v) = \log\vert\mu\vert$ for all $v\in E^l_{\alpha(m_0)}(\theta)\setminus\{0\}$. Hence, the multiplicity of $\log\vert\mu\vert$ for $\chi_{\alpha(m_0)}$ is one for all $\theta\in{\mathbb T}^n$.

By Theorem \ref{Lyapunov}, the $\alpha(m_0)$-invariant one-dimensional distribution $E$ given by $E(\theta) = {\rm Span}(X_\Phi(\theta))$ satisfies $\chi_{\alpha(m_0)}(\theta,X_\Phi(\theta))=\log\vert\mu\vert$ for all $\theta\in{\mathbb T}^n$. If $E(\theta)\ne E^l_{\alpha(m_0)}(\theta)$ at some $\theta\in{\mathbb T}^n$, then $E(\theta)+E^l_{\alpha(m_0)}(\theta)$ is a two-dimensional subspace of  ${\bf T}_\theta {\mathbb T}^n$ for which $\chi_{\alpha(m_0)}(\theta, v) = \log\vert\mu\vert$ for all $v\in \big(E(\theta)+ E^l_{\alpha(m_0)}(\theta)\big)\setminus\{0\}$. This contradicts the multiplicity of $\log\vert\mu\vert$ for $\chi_{\alpha(m_0)}$ being one at every $\theta$. Thus $E^l_{\alpha(m_0)}(\theta) = E(\theta) = {\rm Span}(X_\Phi(\theta))$ for all $\theta\in{\mathbb T}^n$.

The vector field $h_*X_\Phi$ satisfies $h_*X_\Phi(\theta)\in E^l_B$ for all $\theta\in{\mathbb T}^n$ because ${\rm Span}(X_\Phi(\theta)) = {\bf T}_{h(\theta)} h^{-1}(E^l_B)$ for all $\theta\in{\mathbb T}^n$. Let $\psi$ be the flow for which $X_\psi = h_*X_\Phi$.  Since $E^l_B$ is a one-dimensional invariant subspace of $B$ and $\chi_B(\theta,v) = \log\vert\mu\vert$ for all $v\in E^l_B\setminus\{0\}$, it follows for all $\theta\in{\mathbb T}^n$ that
\[ \Vert B^kX_\psi(\theta)\Vert = \vert \mu\vert^k  \Vert X_\psi(\theta)\Vert {\rm \ for\ all\ }k\in{\mathbb Z}.\]

Hyperbolicity of $B$ implies that there is $\bar\theta\in{\mathbb T}^n$ such that ${\mathcal O}_B(\bar\theta)=\{ B^k(\bar\theta):k\in{\mathbb Z}\}$ is dense in ${\mathbb T}^n$. Since $\alpha(m_0)_*X_\Phi = \mu X_\Phi$ and $X_\psi = h_*X_\Phi$, the matrix $B$ satisfies $BX_\psi =\mu X_\psi B$. Then $B^k X_\psi = \mu^k X_\psi B^k$, and so $X_\psi B^k = \mu^{-k} B^k X_\psi$. Thus,
\[ \Vert X_\psi(B^k\bar\theta)\Vert = \vert \mu\vert^{-k} \Vert B^k X_\psi(\bar\theta)\Vert = \vert\mu\vert^{-k}\vert\mu\vert^k \Vert X_\psi(\bar\theta)\Vert = \Vert X_\psi(\bar\theta)\Vert {\rm \ for\ all\ }k\in{\mathbb Z}.\]
Denseness of ${\mathcal O}_B(\bar\theta)$ and continuity of $X_\psi$ imply that $\Vert X_\psi(\theta)\Vert = \Vert X_\psi(\bar\theta)\Vert$ for all $\theta\in{\mathbb T}^n$. The one-dimensionality of $E^l_B$ to which $X_\psi$ belongs implies that $X_\psi$ is a constant vector. Thus $X_\psi$ is an eigenvector of $B$.

The assumed topological irreducibility of $\alpha$ and the inclusion $\Delta_{h^{-1}}(\alpha({\mathbb Z}^d))\subset {\rm GL}(n,{\mathbb Z})$ imply that $\Delta_{h^{-1}}(\alpha(\mathbb Z^d))$ is algebraically irreducible. Thus there is $B^\prime\in \Delta_{h^{-1}}(\alpha({\mathbb Z}^d))$ with an irreducible characteristic polynomial. Since $BB^\prime = B^\prime B$ and $B^\prime$ has an irreducible characteristic polynomial, the eigenvector $X_\psi$ of $B$ is an eigenvector of $B^\prime$ too. Then there is $\vartheta\in{\mathbb R}^\times$ such that the components of $\vartheta^{-1}X_\psi$ form a ${\mathbb Q}$-basis for an algebraic number field of degree $n$ over ${\mathbb Q}$ (see Propositions 1 and  8 in \cite{WA}). Thus the flow $\phi$ determined by $X_\phi = \vartheta^{-1}X_\psi$ is irrational of Koch type for which $h_*X_\Phi = X_\psi = \vartheta X_\phi$. \hfill $\Box$

\end{document}